\documentclass[preprint]{article}

\usepackage[]{amsfonts} \usepackage[]{amsmath}
\usepackage[]{amssymb} \usepackage[]{latexsym}
\usepackage{graphicx,color} \usepackage{amsthm}
\usepackage{mathrsfs} \usepackage{url}
\usepackage{cancel} \usepackage{enumerate}
 \usepackage{enumitem} \usepackage{subfigure} \usepackage[]{latexsym}
 \usepackage[]{amssymb} \usepackage{mathptmx}

\begin{document}

\numberwithin{equation}{section}

\theoremstyle{plain}
\newtheorem{theorem}{Theorem}[section] \newtheorem*{theorem*}{Theorem}
\newtheorem{proposition}[theorem]{Proposition} \newtheorem*{proposition*}{Proposition}
\newtheorem{lemma}[theorem]{Lemma} \newtheorem*{lemma*}{Lemma}
\newtheorem{corollary}[theorem]{Corollary} \newtheorem*{corollary*}{Corollary}

\theoremstyle{definition}
\newtheorem{definition}[theorem]{Definition} \newtheorem*{definition*}{Definition}
\newtheorem{example}[theorem]{Example} \newtheorem*{example*}{Example}
\newtheorem{remark}[theorem]{Remark} \newtheorem*{remark*}{Remark}
\newtheorem{hypotheses}[theorem]{Hypotheses} \newtheorem{assumption}[theorem]{Assumption}
\newtheorem{notation}[theorem]{Notation}

\newcommand{\ds}{\displaystyle} \newcommand{\nl}{\newline}
\newcommand{\eps}{\varepsilon} \newcommand{\barz}{\overline{z}}
\newcommand{\bE}{\mathbb{E}} \newcommand{\barm}{\overline{m}} \newcommand{\QV}{\mbox{QV}} \newcommand{\barM}{\overline{M}}
\newcommand{\cA}{\mathcal{A}}
\newcommand{\cB}{\mathcal{B}}
\newcommand{\cC}{\mathcal{C}}
\newcommand{\cL}{\mathcal{L}}
\newcommand{\cS}{\mathcal{S}}
\newcommand{\cO}{\mathcal{O}}
\newcommand{\cQ}{\mathcal{Q}}
\newcommand{\cH}{\mathcal{H}}
\newcommand{\cF}{\mathcal{F}}
\newcommand{\fF}{\mathfrak{F}}
\newcommand{\cM}{\mathcal{M}}
\newcommand{\cI}{\mathcal{I}}
\newcommand{\cD}{\mathcal{D}}
\newcommand{\cT}{\mathcal{T}}
\newcommand{\cG}{\mathcal{G}}
\newcommand{\cP}{\mathcal{P}}
\newcommand{\bP}{\mathbb{P}}
\newcommand{\bT}{\mathbb{T}}
\newcommand{\bD}{\mathbb{D}}
\newcommand{\bQ}{\mathbb{Q}}
\newcommand{\bC}{\mathbb{C}}
\newcommand{\bN}{\mathbb{N}}
\newcommand{\bR}{\mathbb{R}}

\title{The Functional Meyer-Tanaka Formula}

\author{Yuri F. Saporito \thanks{Escola de Matem\'atica Aplicada (EMAp), Funda\c{c}\~ao Get\'ulio Vargas (FGV), Rio de Janeiro, Brazil, {\em yuri.saporito@fgv.br}.}}

\maketitle

\begin{abstract}
The functional It\^o formula, firstly introduced by Bruno Dupire for continuous semimartingales, might be extended in two directions: different dynamics for the underlying process and/or weaker assumptions on the regularity of the functional. In this paper, we pursue the former type by proving the functional version of the Meyer-Tanaka Formula. Following the idea of the proof of the classical time-dependent Meyer-Tanaka formula, we study the mollification of functionals and its convergence properties. As an example, we study the running maximum and the max-martingales of Yor and Ob{\l}\'oj.
\end{abstract}

\section{Introduction}

Our goal in this article is to prove the functional extension of the well-known Meyer-Tanaka formula. The theory of functional It\^o calculus was presented in the seminal paper \cite{fito_dupire} and it was further developed and applied to diverse topics, for instance, in \cite{fito_zhang_viscosity_I, fito_zhang_viscosity_II, fito_peng_bsde, fito_saporito_greeks, fito_cont_martingales, fito_cont_change, fito_generalization}. Before proceeding, a remark regarding nomenclature. In this paper, the adjective \textit{classical} will always refer to the finite-dimensional It\^o stochastic calculus.

The Meyer-Tanaka formula is the extension of It\^o formula to convex functions. More precisely, in the classical case, if $f:\bR \longrightarrow \bR$ is convex and $(x_t)_{t \geq 0}$ is a continuous semimartingale, then
\begin{align}
f(x_t) = f(x_0) + \int_0^t f'(x_s) dx_s + \int_{\bR} L^x(t,y) df'(y), \label{eq:tanaka}
\end{align}
where $f'$ is the left-derivative of $f$ and $L^x(s,y)$ is the local time of the process $x$ at $y$; see \cite{Karatzas1988}, for example. This formula is easily generalized to functions $f$ that are absolutely continuous with derivative of bounded variation, which is equivalent to say that $f$ is the difference of two convex functions. We would like to remind the reader that the local time is defined by the limit in probability:
$$L^x(t,y) = \lim_{\eps \to 0^+} \frac{1}{4\eps} \int_0^t 1_{[y-\eps,y+\eps]}(x_s) d\langle x \rangle_s,$$
where $\langle x \rangle$ is the quadratic variation of the process $x$. We are adhering the convention $4\eps$ instead of $2\eps$. The random field $(L^x(t,y))_{t, y}$ is a.s continuous and increasing in $t$ and c\`adl\`ag in $y$. The following extension to time-dependent functions was established in \cite{generalized_ito_elworthy}:
\begin{align}
f(t,x_t) &= f(0,x_0) + \int_0^t \partial_t^- f(s,x_s) ds + \int_0^t \partial_x^- f(s,x_s) dx_s \label{eq:time_depent_tanaka} \\
&+ \int_{\bR} L^x(t,y) d_y \partial_x^- f(t,y) - \int_{\bR} \int_0^t L^x(s,y) d_{s,y} \partial_x^- f(s,y), \nonumber
\end{align}
where $\partial_t^- f$ and $\partial_x^- f$ are the time and space left-derivatives, respectively. It is assumed that $f$ is absolutely continuous in each variable, $\partial_t^- f$ and $\partial_x^- f$ exist, are left-continuous and locally bounded, $\partial_x^- f$ is of locally bounded variation in $\bR_+ \times \bR$ and $\partial_x^- f(0,\cdot)$ is of locally bounded variation in $\bR$. The notation $d_y$ and $d_{s,y}$ mean integration with respect to the $y$ variable and the $(s,y)$ variables, respectively. We forward the reader to the reference cited above for some other different generalizations of Meyer-Tanaka formula (\ref{eq:tanaka}) and for the precise definition of the Lebesgue-Stieltjes integral $\int_{\bR} \int_0^t L^x(s,y) d_{s,y} \partial_x^- f(s,y)$.

Since a functional extension of the Meyer-Tanaka would be inherently time\- dependent, Equation (\ref{eq:time_depent_tanaka}) is of utmost importance for our goal. However, we will not pursue a functional extension of (\ref{eq:time_depent_tanaka}) in its full generality of assumptions. It is clear that some of the technical assumptions of the results presented in our work could be weakened along the lines of \cite{generalized_ito_elworthy}, but in order to provide a clear exposition of the subject we will consider technical assumptions that are general enough to introduce the important techniques without adding a cumbersome notation.

There are several other generalizations of the It\^o formula that could be extended to the functional framework, for instance, \cite{Al-Hussaini1987,Peskir2005,lowther_nondiff,Ghomrasni2003,generalized_ito_elworthy,Russo1996,Foellmer1995, generalized_ito_feng, Carlen1992}. We will not pursue them here, of course, but we hope that the foundations laid in this work might help in this task.

Meyer-Tanaka formula and its generalizations have many interesting applications in Finance, as, for instance, \cite{local_time_barrier,russian_option, american_min_detemple}. Other applications can be found in the theory of Local Volatility of \cite{dupire94}, see for example \cite{klebaner02}.

The paper is organized as follows: we finish this introduction with a presentation of functional It\^o calculus and we define the mollification of functionals in Section \ref{sec:mollification}. This is a very important tool that will be used in Section \ref{sec:tanaka_formula} in order to prove the functional extension of the Meyer-Tanaka formula. In Section \ref{sec:applications}, we will apply the theory to the running maximum to find a pathwise version of a famous identity by Paul L\'evy and we will also study the max-martingales of Yor and Ob{\l}\'oj in the light of the functional It\^o calculus.

\subsection{A Brief Primer on Functional It\^o Calculus}

In this section we will present a short review of the functional It\^o calculus introduced in \cite{fito_dupire}. The goal is to familiarize the reader with the notation, main definitions and theorems needed for the results that follow.

The space of c\`adl\`ag paths in $[0,t]$ will be denoted by $\Lambda_t$. For a fixed time horizon $T > 0$, we define the \textit{space of paths} as
$$\Lambda = \bigcup_{t \in [0,T]} \Lambda_t.$$

We will denote elements of $\Lambda$ by upper case letters and often the final time of its domain will be subscripted, e.g. $Y \in \Lambda_t \subset \Lambda$ will be denoted by $Y_t$. The value of $Y_t$ at a specific time will be denoted by lower case letters: $y_s = Y_t(s)$, for any $s \leq t$. Moreover, if a path $Y_t$ is fixed, the path $Y_s$, for $s \leq t$, will denote the restriction of the path $Y_t$ to the interval $[0,s]$.

The following important path deformations are always defined in $\Lambda$. For $Y_t \in \Lambda$ and $t \leq s \leq T$, the \textit{flat extension} of $Y_t$ up to time $s \geq t$ is defined as
$$Y_{t,s-t}(u) = \left\{
\begin{array}{ll}
  y_u,  &\mbox{ if } \quad 0 \leq u \leq t, \\
  y_t,  &\mbox{ if } \quad t \leq u \leq s,
\end{array}
\right.$$
see Figure \ref{fig:flat}. For $h \in \bR$, the \textit{bumped path}, see Figure \ref{fig:bump}, is defined by
$$Y_t^h(u) = \left\{
\begin{array}{ll}
  y_u,      &\mbox{ if } \quad 0 \leq u < t, \\
  y_t + h,  &\mbox{ if } \quad u = t.
\end{array}
\right.$$

\begin{figure}[h!]
\centering
  \begin{minipage}[b]{0.4\linewidth}
    \centering
    \includegraphics[width=0.8\linewidth]{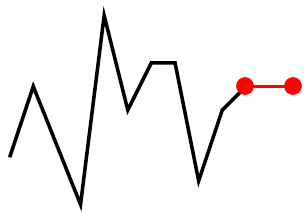}
    \caption{Flat extension of a path.}
    \label{fig:flat}
  \end{minipage}
  \hspace{0.5cm}
  \begin{minipage}[b]{0.4\linewidth}
    \centering
    \includegraphics[width=0.8\linewidth]{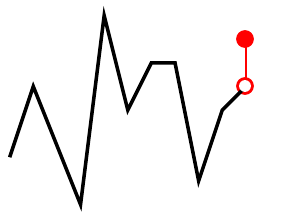}
    \caption{Bumped path.}
    \label{fig:bump}
  \end{minipage}
\end{figure}

For any $Y_t, Z_s \in \Lambda$, where it is assumed without loss of generality that $s \geq t$, we consider the following metric in $\Lambda$,
$$d_{\Lambda}(Y_t,Z_s) = \| Y_{t,s-t} - Z_s\|_{\infty} + |s -t|,$$
where
$$\|Y_t\|_{\infty} = \sup_{u \in [0,t]} |y_u|.$$
One could easily show that $(\Lambda, d_{\Lambda})$ is a complete metric space.

Additionally, a \textit{functional} is any function $f:\Lambda \longrightarrow \bR$. Continuity with respect to $d_{\Lambda}$ is defined as the usual definition of continuity in a metric space and is denominated $\Lambda$\textit{-continuity}.

For a functional $f$ and a path $Y_t$ with $t < T$, the \textit{time functional derivative} of $f$ at $Y_t$ is defined as
\begin{align}
\Delta_t f(Y_t) = \lim_{\delta t \to 0^+} \frac{f(Y_{t,\delta t}) - f(Y_t)}{\delta t}, \label{eq:time_deriv}
\end{align}
whenever this limit exists. The \textit{space functional derivative} of $f$ at $Y_t$ is defined as, if the limit exists,
\begin{align}
\Delta_x f(Y_t) = \lim_{h \to 0} \frac{f(Y_t^h) - f(Y_t)}{h}. \label{eq:space_deriv}
\end{align}

Finally, for any $i,j \in \{0\} \cup \bN \cup \{+\infty\}$, a functional $f: \Lambda \longrightarrow \bR$ is said to belong to $\bC^{i,j}$ if it is $\Lambda$-continuous and it has $\Lambda$-continuous derivatives $\Delta_t^{(k)} f$ and $\Delta_x^{(m)} f$, for $k=1,\ldots, i$ and $m=1,\ldots,j$. Here, clearly, $\Delta_t^{(k)} = \Delta_t (\Delta_t^{(k-1)})$ and $\Delta_x^{(m)} = \Delta_x (\Delta_x^{(m-1)})$. Moreover, we use the notation $\Delta_{xx} = \Delta_x^{(2)}$.\\

The attentive reader might have noticed that we have not introduced any probability notation so far. We start by fixing a probability space $(\Omega, \fF, \bP)$. We state now the functional It\^o formula. The proof can be found in \cite{fito_dupire}.
\begin{theorem}[Functional It\^o Formula; \cite{fito_dupire}]\label{thm:fif}
Let $x$ be a continuous semimartingale and $f \in \bC^{1,2}$. Then, for any $t \in [0,T]$,
$$f(X_t) = f(X_0) + \int_0^t \Delta_t f(X_s) ds + \int_0^t \Delta_x f(X_s) dx_s + \frac{1}{2} \int_0^t \Delta_{xx} f(X_s) d\langle x\rangle_s \quad \bP-\mbox{a.s.}$$
\end{theorem}

One should notice that the It\^o formula above is of the same form as the classical It\^o formula for continuous semimartingale, the only change being the definition of the time and space functional derivatives given by Equations (\ref{eq:time_deriv}) and (\ref{eq:space_deriv}). This theorem was extended in terms of weakening the regularity of $f$ and generalizing the dynamics of $x$, see \cite{fito_cont_martingales, fito_cont_change, fito_generalization}. Here, we will examine a different class of functionals than it was considered in these previous works. We now state the main result of this paper:

\begin{theorem*}[Functional Meyer-Tanaka Formula]
Consider a functional $f: \Lambda \longrightarrow \bR$ satisfying Hypotheses \ref{hypo:hypotheses} and let $x$ be a continuous semimartingale. Then
\begin{align}
f(X_t) &= f(X_0) + \int_0^t \Delta_t f(X_s) ds + \int_0^t \Delta_x^- f(X_{s}) dx_s \label{eq:meyer_tanaka_form_intro} \\
&+ \int_{\bR}  L^x(t,y) d_y\partial_y^- f(X_{t-}^y) - \int_0^t \int_{\bR} L^x(s,y) d_{s,y}\partial_y^- f(X_{s-}^y) \quad \bP-\mbox{a.s.}, \nonumber
\end{align}
where $X_{s-}^y$ is the path $X_s$ with the value at $s$ substituted by $y$, see Equation (\ref{eq:cF}).
\end{theorem*}

\begin{notation}
$d_y \phi(y)$ and $d_{s,y} \phi(s,y)$ denote the Lebesgue-Stieltjes integration with respect to the integrator $\phi(y)$ and $\phi(s,y)$, respectively.
\end{notation}

The main example of non-smooth functional to have in mind is the running maximum:
\begin{align}\label{eq:running_max}
\barm(Y_t) = \sup_{0 \leq s \leq t} y_s.
\end{align}
For more details on this functional, we forward the reader to Section \ref{sec:running_max}

\section{Functional Mollification}\label{sec:mollification}

In this section, we investigate the mollification of functionals. The goal is to create a sequence of smooth functionals converging to the original one in various senses. This technique will be used to prove the functional Meyer-Tanaka formula as it is similarly done in the proof of its classical version.

\begin{definition}\label{def:capital_F}
For any functional $f: \Lambda \longrightarrow \bR$, we define $F: \Lambda \times \bR \longrightarrow \bR$ as
\begin{align}
F(Y_t,h) = f(Y_t^h).\label{eq:capital_F}
\end{align}
When denoting functionals, capital letters will be used as above, i.e. it will denote a function with domain $\Lambda \times \bR$ where the first variable is the path and the second variable is the bump applied to this path. This notation will be carried out in the remainder of the paper. We choose to use this notation to help the analysis of the space functional derivative of the mollification.
\end{definition}

A mollifier in $\bR$ is a positive function $\rho : \bR \longrightarrow [0,+\infty)$ such that $\rho \in C_c^{\infty}(\bR)$, the space of compactly supported smooth functions; $\int_{\bR} \rho(z) dz = 1$; and $\rho_n(x) := n \rho(n x)$ converges to Dirac delta in the sense of distributions. We also refer to the sequence $(\rho_n)_{n \in \bN}$ as the mollifiers. Notice that $\rho_n \in C^{\infty}_c(\bR)$.


\begin{definition}\label{def:mollification}
The sequence of \textit{mollified functionals} is defined as
\begin{align}
F_n(Y_t,h) &= \int_{\bR} \rho_n(h - \xi) F(Y_t,\xi) d\xi = \int_{\bR} \rho_n(\xi) F(Y_t,h-\xi) d\xi. \label{eq:mollifier_R} 
\end{align}
\end{definition}

\begin{remark}
This mollification is well-defined as long as the real function $F(Y_t,\cdot)$ is locally integrable for any path $Y_t \in \Lambda$. See \cite{pde_evans}, for instance, for details on the mollification in the case of real functions.
\end{remark}

\begin{proposition}
Suppose $f$ is $\Lambda$-continuous. Then $F(Y_t,\cdot)$ is continuous for each $Y_t \in \Lambda$, $F_n$ is well-defined and, as a functional, is infinitely differentiable in space. Moreover,
$$\Delta_x^{(k)} F_n(Y_t,h) = \partial_h^{(k)} F_n(Y_t,h),$$
where $\partial_h^{(k)}$ denotes the $k$-th derivative with respect to $h$. This is the main property of the mollified functionals.
\end{proposition}

\begin{proof}
Notice that since the functional $f$ is $\Lambda$-continuous,  $F(Y_t,\cdot)$ is then continuous for fixed $Y_t \in \Lambda$, because $d_{\Lambda}(Y_t^{h_1}, Y_t^{h_2}) = |h_1 - h_2|$. This implies $F(Y_t,\cdot)$ is locally integrable, and therefore the mollification $F_n$ is well-defined. Notice now that $F(Y_t^z,h) = F(Y_t,h+z)$ and then
\begin{align}
F_n(Y_t^z,h) &= \int_{\bR} \rho_n(h - \xi) F(Y_t,\xi+z) d\xi \label{eq:bump_mollification}\\
&= \int_{\bR} \rho_n(h - (\xi - z)) F(Y_t,\xi) d\xi = F_n(Y_t, h+z). \nonumber
\end{align}
Thus, for any $k \in \bN$,
$$\Delta_x^{(k)} F_n(Y_t,h) = \partial_h^{(k)} F_n(Y_t,h).$$
\end{proof}

We would like also to point it out that a particular mollification of the running maximum was considered in \cite{fito_dupire} to derive a pathwise version of the famous formula due to L\'evy:
$$\max_{0 \leq s \leq t} x_s = x_0 + L^{x-\barm}\left(t,0\right),$$
where $\barm$ is the running maximum process and $x$ is a continuous semimartingale. The reader is forwarded to \cite[Chapter 3 and Chapter 6]{Karatzas1988} for more details on results regarding the relations between local time and the running maximum in the Brownian motion case.

\subsection{$\Lambda$-Continuity of the Mollified Functionals and its Derivatives}

In this section we will study the relation of continuity of $f$ and of its mollification $F_n$.

We have already seen that, if $F(Y_t,\cdot)$ is locally integrable for any given $Y_t \in \Lambda$, then $F_n(Y_t,\cdot)$ is infinitely differentiable in $\bR$, and therefore it is also continuous. However, differentiability in the functional sense does not imply $\Lambda$-continuity. Hence, it is necessary to consider a slightly stronger assumption on the continuity of the functional $f$ in order to be able to conclude the $\Lambda$-continuity of $F_n$. We will thus consider the following stronger criterion:

\begin{definition}\label{defi:assump_continuity}

We say that $f$ is \textit{$\Lambda$-$\phi$-equicontinuous} if there exists $\phi: \bR \longrightarrow \bR$ positive and locally integrable depending only on $f$ such that $\forall \ \eps > 0, \ \forall \ Y_t \in \Lambda, \ \exists \ \delta > 0$,
\begin{align}\label{eq:d_Lambda_equicont}
\ d_{\Lambda}(Y_t, Z_s) < \delta \Rightarrow |F(Y_t, \xi) - F(Z_s, \xi)| < \eps \phi(\xi), \ \forall \ \xi \in \bR.
\end{align}
Notice that $\Lambda$-$\phi$-equicontinuity implies that $f$ is $\Lambda$-continuous. Moreover, if $\phi \equiv 1$, then the family of functionals $\{F(\cdot,\xi)\}_{\xi \in\bR}$ is $\Lambda$-equicontinuous.

\end{definition}

The weakening of this assumption could be pursued, but it is not in the scope of this work.

\begin{proposition}\label{prop:continuity_mollifier}
Suppose $f$ is $\Lambda$-$\phi$-equicontinuous. Then, for any $n \in \bN$ and $h \in \bR$, $F_n(\cdot,h)$ and $\Delta_x^{(k)} F_n(\cdot,h)$ are $\Lambda$-continuous, for any $k \in \bN$.
\end{proposition}

\begin{proof}
By Equation (\ref{eq:mollifier_R}), we see
\begin{align*}
|F_n(Y_t,h) - F_n(Z_s,h)| \leq \int_{\bR} \rho_n(h-\xi)|F(Y_t,\xi) - F(Z_s,\xi)| d\xi.
\end{align*}
Hence, fixing $\eps > 0$, $n \in \bN$ and $h \in \bR$, and choosing $\delta > 0$ from the $\Lambda$-$\phi$-equicontinuity of $f$ with $\eps$ equals
$$\frac{\eps}{\int_{\bR} \rho_n(h-\xi) \phi(\xi) d\xi},$$
we have, for $Y_t, Z_s \in \Lambda$ satisfying $d_{\Lambda}(Y_t, Z_s) < \delta$,
\begin{align*}
|F_n(Y_t,h) - F_n(Z_s,h)| &\leq \int_{\bR} \rho_n(h-\xi)|F(Y_t,\xi) - F(Z_s,\xi)| d\xi < \eps.
\end{align*}
Therefore, we conclude that $F_n(\cdot,h)$ is $\Lambda$-continuous for any $n \in \bN$ and $h \in \bR$. Considering now the derivatives of $F_n$, we see
\begin{align*}
\Delta_x^{(k)} F_n(Y_t,h) = \partial_h^{(k)} F_n(Y_t,h) = \int_{\bR} \partial_h^{(k)}(\rho_n(h-\xi)) F(Y_t,\xi) d\xi,
\end{align*}
and since $\partial_h^{(k)}(\rho_n(h-\cdot))$ are in $C^{\infty}_c(\bR)$, the same argument employed above for the $\Lambda$-continuity of $F_n$ can be used to conclude the $\Lambda$-continuity of $\Delta_x^{(k)} F_n(\cdot,h)$.
\end{proof}

\subsection{The Issue with the Time Derivative}\label{sec:time_derivative}

As we have seen, the functional $F_n$ is smooth with respect to the space variable. In this section, we will study the question of the existence of the  time functional derivative. Notice that
$$F_n(Y_{t,\delta t},h) = \int_{\bR} \rho_n(h - \xi) F(Y_{t,\delta t},\xi) d\xi.$$
When is $F_n$ time functional differentiable as in Equation (\ref{eq:time_deriv})?


\begin{definition}\label{def:time_derivative_mollifier}
We say a functional $f$ is \textit{$h$-time functional differentiable} if $F(\cdot, h)$ is time functional differentiable for every $h \in \bR$, i.e.
\begin{align}
\Delta_t F(Y_t, h) = \lim_{\delta t \to 0^+} \frac{F(Y_{t,\delta t},h) - F(Y_t,h)}{\delta t} = \lim_{\delta t \to 0^+} \frac{f((Y_{t,\delta t})^h) - f(Y_t^h)}{\delta t}, \label{eq:time_derivative_mollifier}
\end{align}
exists for every $Y_t \in \Lambda$ and $h \in \bR$.
\end{definition}

We are then ready to answer the previous question:

\begin{proposition}\label{prop:time_derivative_mollifier}
If $f$ is $h$-time functional differentiable, then $\Delta_t F_n$ exists,
$$\Delta_t F_n(Y_t,h) = (\Delta_t F)_n(Y_t,h),$$
for any $Y_t \in \Lambda$ and $h \in \bR$, and
$$\Delta_x^{(k)} \Delta_t F_n(Y_t,h) = \partial_h^{(k)} (\Delta_t F)_n(Y_t,h).$$
Moreover, if $\Delta_t F$ is $\Lambda$-$\phi$-equicontinuous, then $\Delta_t F_n(\cdot,h)$ is $\Lambda$-continuous, and hence in $\bC^{0, \infty}$.
\end{proposition}

\begin{proof}
For fixed $Y_t \in \Lambda$ and $\xi \in \bR$, define $\psi(\delta t) = F(Y_{t,\delta t}, \xi)$. By Definition \ref{def:time_derivative_mollifier}, $\psi \in C^1(\bR_+)$ and therefore,
$$\Delta_t F_n(Y_t,h) = \int_{\bR} \rho_n(h - \xi) \Delta_t F(Y_t,\xi) d\xi = (\Delta_t F)_n(Y_t, \xi).$$
Since $\Delta_t F$ is $\Lambda$-$\phi$-equicontinuous and $\Delta_t F_n(Y_t,h) = (\Delta_t F)_n(Y_t, \xi)$, by Proposition \ref{prop:continuity_mollifier}, we conclude that $\Delta_t F_n$ is $\Lambda$-continuous. 
\end{proof}

\begin{remark}\label{rmk:lie_bracket}
We would like to point out the similarity of the Equation (\ref{eq:time_derivative_mollifier}) and the limit characterization of the Lie bracket given in \cite[Lemma 3.2]{fito_saporito_greeks}:
\begin{align*}
[\Delta_t, \Delta_x] f(Y_t) = \lim_{\delta t \to 0^+ \atop h \to 0} \frac{ f((Y_{t,\delta t})^h) - f((Y_t^h)_{t,\delta t})}{h \delta t}. 
\end{align*}
However, it is obvious that Definition \ref{def:time_derivative_mollifier} \textit{does not require} the functional $f$ to be locally weakly path-dependent ($[\Delta_t, \Delta_x] f = 0$, as defined in \cite{fito_saporito_greeks}). Definition \ref{def:time_derivative_mollifier} is indeed just a technicality and encompasses many interesting functionals. For example, the running integral ($f(Y_t) = \int_0^t y_s ds$) satisfies Assumption \ref{def:time_derivative_mollifier} and it is not locally weakly path-dependent.
\end{remark}

\subsubsection{Time and Joint Mollification}

We will not pursue this here, but it is important to mention two different mollification possibilities:\\ \\
\textbf{(i) Time Mollification}:
\begin{align}
F_n(Y_t,\delta t) &= \int_{\bR} \rho_n(\delta t - \eta) f(Y_{t,\eta}) d\eta. \label{eq:time_mollifier}
\end{align}
\textbf{(ii) Joint Mollification}:
\begin{align}
F_n(Y_t,\delta t,h) &= \int_{\bR} \int_{\bR} \rho_n(h - \xi) \rho_n(\delta t - \eta) f((Y_{t,\eta})^{\xi}) d\xi d\eta. \label{eq:joint_mollifier}
\end{align}

An obvious issue with the joint mollification is the choice between $f((Y_{t,\eta})^{\xi})$ and $f((Y_t^{\xi})_{t,\eta})$; both would be initially valid choices. This is not a problem when we restrict ourselves to the path-independent case: $f(Y_t) = h(t,y_t)$. However, as it was noted in \cite{fito_saporito_greeks}, the different ordering of bump and flat extension is a very important aspect of the functional It\^o calculus.

Additionally, as it happened in the aforesaid reference in a different circumstance, the Lie bracket of the operators $\Delta_t$ and $\Delta_x$ would probably play an important role if the joint mollification were used. 

\subsection{Integration by Parts} \label{sec:int_by_parts}

We will now derive some integration by parts computations that will be useful later in the proof of the functional Meyer-Tanaka formula.\\

First some definitions. For any $Y_t \in \Lambda$ and $y \in \bR$,
\begin{align}\label{eq:path_change_last_value}
Y_{t-}^y(u) = \left\{
\begin{array}{ll}
  y_u,      &\mbox{ if } \quad 0 \leq u < t, \\
  y,  &\mbox{ if } \quad u = t.
\end{array}
\right.
\end{align}
Notice that $Y_{t-}^y = Y_t^{y - y_t} \in \Lambda_t$ and it is different than $(Y_{t-})^y = Y_t^{y - y_t + y_{t-}}$. Moreover, define
\begin{align}
\cF(Y_t,y) = f(Y_{t-}^y). \label{eq:cF}
\end{align}

The definition of the function $\cF$ above serves two purposes. Firstly, alleviates notation. Secondly, it helps us take derivatives with respect to the $Y_t$ and the last value $y$ separately. Capital calligraphic letters will always be used as above meaning that it will denote a function with domain $\Lambda \times \bR$ where the first variable is the path and the second variable is the value will replace the last value of the path. We will keep this notation through out the paper.\\

We start by noticing that, for any function $q: \bR^2 \longrightarrow \bR$ regular enough for the computations to follow, the subsequent identity is obviously true:
\begin{align*}
\int_{\bR} \left( \int_0^t \cF(Y_s,y) d_s q(s,y) \right) dy &= \int_{\bR} \cF(Y_t,y) q(t,y) dy \\
&- \int_{\bR} \left( \int_0^t \partial_t \cF(Y_s,y) q(s,y)ds \right) dy,
\end{align*}
where $\cF$ is given by Equation (\ref{eq:cF}) and
$$\partial_t \cF(Y_s,y) = \lim_{u \to s} \frac{\cF(Y_s,y) - \cF(Y_u,y)}{s-u},$$
the usual time derivative of a function. Let us now verify that this derivative exists under certain regularity assumptions. Notice that $\cF(Y_s,y)$ does not depend on the last value of the path $Y_s$, and hence $\Delta_x \cF(Y_s,y) = \Delta_{xx} \cF(Y_s,y) = 0$. So, if $f$ satisfies Definition \ref{def:time_derivative_mollifier}, $\Delta_t \cF(Y_s,y)$ exists. Assuming $Lambda$-continuity of $\cF(\cdot, y)$ and $\Delta_t \cF(\cdot, y)$ implies that $\cF(\cdot, y) \in \bC^{1,2}$.  Then, one can show, by the functional It\^o formula, Theorem \ref{thm:fif}, that for any continuous semimartingale $x$,
$$\cF(X_s,y) = \cF(X_u,y) + \int_u^t \Delta_t \cF(X_r,y)dr,$$
which implies that
$$\partial_t \cF(X_s,y) = \Delta_t \cF(X_s,y), \ \forall \ s \geq 0, \quad \bP-\mbox{a.s.}$$
Moreover,
\begin{align}
\cF(Y_t,y+h)  = F(Y_{t-}^y,h) = f(Y_{t-}^{y+h}),\label{eq:relation_f_F_cF}
\end{align}
and then
$$\partial_y^{(k)} \cF(Y_t,y) = \partial_h^{(k)} F(Y_{s-}^y,0) = \Delta_x^{(k)} f(Y_{s-}^y) .$$

Before proceeding, we would like to comment on the commutation of $\partial_t$ and $\partial_y$. It is well-known now that $\Delta_t$ and $\Delta_x$ do not commute. However, we do not experience a similar problem here. $\partial_t$ and $\partial_y$ do commute: $\partial_y \partial_t \cF(Y_t,y) = \partial_t \partial_y \cF(Y_t,y)$, as one can easily verify by direct computation and assuming these derivatives exist and are continuous. The reason is that in the definition of $\cF(Y_t, y)$ it is implied that the bump $y$ happens always at the end of the path $Y_t$. Therefore, there is no ambiguity in the order of the time perturbation and the bump that we experience in the case of $\Delta_x$ and $\Delta_t$.\\

If $g \in C^1_c(\bR)$ and $\partial_{yy} \cF(Y_t,y)$ exists, then
\begin{align}
\int_{\bR}  g(y) \Delta_{xx} f(Y_{t-}^y) dy &= \int_{\bR} g(y) \partial_{yy} \cF(Y_t,y) dy \label{eq:IBP_second_der_measure}\\
&=  -\int_{\bR} g'(y) \partial_y \cF(Y_t,y) dy \nonumber
\end{align}
Furthermore, if we consider $q: \bR^2 \longrightarrow \bR$ smooth with compact support and assume $\partial_t \partial_{yy} \cF(Y_t,y)$ exists, we find
\begin{align*}
&\int_0^t \int_{\bR}  \partial_t \Delta_{xx} f(Y_{s-}^y) q(s,y) dy ds = \int_0^t \int_{\bR} \partial_t \partial_{yy} \cF(Y_s,y) q(s,y)  dyds\\
&= \int_0^t \int_{\bR} \partial_y \partial_t \partial_y \cF(Y_s,y) q(s,y) dy ds \\
&= - \int_0^t \int_{\bR} \partial_t \partial_y \cF(Y_s,y) \partial_y q(s,y) dy ds \\
&= -  \left. \int_{\bR} \partial_y \cF(Y_s,y) \partial_y q(s,y) \right|_0^t dy + \int_0^t \int_{\bR} \partial_y \cF(Y_s,y)  \partial_{sy} q(s,y) dy ds,
\end{align*}
where
$$\left. \vphantom{\int} \psi(s,y) \right|_0^t = \psi(t,y) - \psi(0,y).$$

\section{Functional Meyer-Tanaka Formula}\label{sec:tanaka_formula}

\subsection{Local Time}

The local time of the process $x$ at level $y$, denoted by $L^x(s,y)$, is defined as the limit in probability:
$$L^x(t,y) = \lim_{\eps \to 0^+} \frac{1}{4\eps} \int_0^t 1_{[y-\eps,y+\eps]}(x_s) d\langle x \rangle_s.$$
A very important identity related to the local time is the \textit{occupation times formula}, \cite[Corollary 1.6, Chapter VI]{Revuz2004}, which says that if $\varphi : \bR_+ \times \bR \longrightarrow \bR$ is bounded and measurable, then
\begin{align}
\int_0^t \varphi(s,x_s) d\langle x \rangle_s = 2 \int_{\bR} \left( \int_0^t \varphi(s,y) d_s L^x(s,y) \right) dy, \ \forall \ t \geq 0, \quad \bP-\mbox{a.s.} \label{eq:occup_time}
\end{align}

The following extension of the occupation time formula will be fundamental in the following, see \cite[Exercise 1.15, Chapter VI]{Revuz2004}.

\begin{lemma}
For any bounded measurable function $\psi : \bR_+ \times \Omega \times \bR \longrightarrow \bR$,
\begin{align}
\hspace{-0.1cm}\int_0^t \psi(s,\omega,x_s) d\langle x \rangle_s = 2 \int_{\bR} \left( \int_0^t \psi(s,\omega,y) d_s L^x(s,y) \right) dy, \ \forall \ t \geq 0, \quad \bP-\mbox{a.s. } \label{eq:occup_time_classical}
\end{align}

\end{lemma}  

\begin{proof}
By Equation (\ref{eq:occup_time}), for any $\varphi : \bR_+ \times \bR \longrightarrow \bR$ bounded and measurable, there exists $\Omega_\varphi \in \fF$ with $\bP(\Omega_\varphi) = 1$ such that, for each $\omega \in \Omega_\varphi$,
$$\int_0^t \varphi(s,x_s(\omega)) d\langle x \rangle(\omega)_s = 2 \int_{\bR} \left( \int_0^t \varphi(s,y) d_s L^x(\omega)(s,y)\right) dy, \ \forall \ t \geq 0,$$
where $\langle x \rangle(\omega)$ and $L^x(\omega)$ are the realizations of the quadratic variation and the local time, respectively. Moreover, since $\psi(\cdot,\omega,\cdot)$ is bounded and measurable, it can be uniformly approximated by simple functions of the form:
$$\psi_n(t,\omega,y) = \sum_{k=1}^{b_n} a_{k,n}(t,y) 1_{A_{k,n}}(\omega).$$
Define now $\Omega_\psi = \bigcap_{n=1}^{+\infty} \bigcap_{k=1}^{b_n}\Omega_{k,n}$, where $\Omega_{k,n}$ is defined as $\Omega_\varphi$ for $\varphi = a_{k,n}$. Therefore, $\bP(\Omega_\psi) = 1$ and, by the occupation time formula for functions on $\bR_+ \times \bR$, we find, for $\omega \in \Omega_\psi$,
\begin{align*}
\int_0^t \psi_n(s,\omega,x_s(\omega)) d\langle x(\omega) \rangle_s &= \sum_{k=1}^{b_n} 1_{A_{k,n}}(\omega) \int_0^t a_{k,n}(s,x_s(\omega)) d\langle x(\omega) \rangle_s \\
&= \sum_{k=1}^{b_n} 1_{A_{k,n}}(\omega) 2 \int_{\bR} \left( \int_0^t a_{k,n}(s,y) d_s L^x(\omega)(s,y) \right) dy \\
&= 2 \int_{\bR} \left( \int_0^t \psi_n(s,\omega,y) d_s L^x(\omega)(s,y) \right) dy.
\end{align*}
Letting $n \to +\infty$ and using the uniformity of the convergence $\psi_n \to \psi$, we have found the desired result.
\end{proof}

For a given functional $f$, we would like to apply the proposition above to $\psi_f(s,\omega,y) = \cF(X_s(\omega),y)$, where $\cF$ is defined in Equation (\ref{eq:cF}). Then, for every functional $f$ such that $\psi_f$ above is bounded and measurable, we have
\begin{align}
\int_0^t f(X_s) d\langle x \rangle_s =  2 \int_{\bR} \left( \int_0^t \cF(X_s,y) d_s L^x(s,y) \right) dy. \label{eq:occup_time_functional}
\end{align}

\begin{example}[Running Integral]

Consider the running integral functional $f(Y_t) = \int_0^t y_u du$. We clearly have $\cF(Y_s,y) = f(Y_s)$ and moreover, we find
\begin{align*}
&\bullet \ \int_0^t f(X_s) d\langle x \rangle_s = \int_0^t \left(\int_0^s x_u du \right) d\langle x \rangle_s = \int_0^t (\langle x \rangle_t - \langle x \rangle_u) x_u du, \\ \\
&\bullet \ \int_{\bR} \left( \int_0^t \cF(X_s,y) d_s L^x(s,y) \right) dy = \int_{\bR} \left( \int_0^t \left(\int_0^s y_u du \right) d_s L^x(s,y) \right) dy \\
& \hskip 3.5cm  = \int_0^t \left(\int_{\bR} L^x(t,y) dy - \int_{\bR} L^x(u,y)dy \right)x_u du.
\end{align*}
Therefore, since
$$\langle x \rangle_t = 2 \int_{\bR} L^x(t,y) dy,$$
we verify Equation (\ref{eq:occup_time_functional}) for this particular example.

\end{example}

\subsection{Convergence Properties} \label{sec:conv_properties}

The idea behind the proof of the classical Meyer-Tanaka formula (see \cite{Karatzas1988} and \cite{generalized_ito_elworthy}, for example) is to apply It\^o formula to the smooth mollification of the function in consideration, let $n$ go to infinity to approximate the original function and then analyze the limit of all the terms of the It\^o formula. Having this strategy in mind, in this section we will investigate the convergence of certain quantities that will be important when proving the functional Meyer-Tanaka formula.

We firstly define the functional $f_n :\Lambda \longrightarrow \bR$ as
\begin{align}
f_n(Y_t) = F_n(Y_t,0) = \int_\bR \rho_n(\xi) F(Y_t,-\xi) d\xi, \label{eq:f_n}
\end{align}
where $F_n$ is the mollification of $F$ given in Equation (\ref{eq:mollifier_R}). 

\begin{remark}\label{rmk:bound_mollifier}

In what follows, we will explicitly use the fact that the mollifier $\rho$ has compact support. Without loss of generality, we may assume that its support is inside $[\rho_{\min}, \rho_{\max}]$, where $\rho_{\min} < 0$ and $\rho_{\max} > 0$.

\end{remark}

\begin{proposition}\label{prop:conv_properties}
\begin{enumerate}[leftmargin=0cm,itemindent=.5cm,labelwidth=\itemindent,labelsep=0cm,align=left]

Assume $\partial_h^- F(Y_t, \cdot)$ exists. The following facts hold true:\\

\item For each $Y_t \in \Lambda$, if $F(Y_t,\cdot)$ and $\partial_h^- F(Y_t,\cdot)$ are continuous at 0, then
\begin{align*}
&\lim_{n \to +\infty} f_n(Y_t) = f(Y_t) ,\\
&\lim_{n \to +\infty} \Delta_x f_n(Y_t) = \lim_{n \to +\infty} \partial_h F_n(Y_t,0) = \partial_h^- F(Y_t,0) = \Delta_x^- f(Y_t).
\end{align*}

\item If $f$ satisfies Definition \ref{def:time_derivative_mollifier}, we have
\begin{align*}
&\lim_{n \to +\infty} \Delta_t f_n(Y_t) = \Delta_t f(Y_t) ,\\
&\lim_{n \to +\infty} \int_0^t \Delta_t f_n(Y_s) ds = \int_0^t \Delta_t f(Y_s) ds,
\end{align*}
for any $Y_t \in \Lambda$.\\

\item If $(\partial_h^- F(X_s,-h))_{s \in [0,T]}$ is bounded in $h \in [\rho_{\min}, \rho_{\max}]$ by an $x$-integrable process, then
$$\lim_{n \to +\infty} \int_0^t \Delta_x f_n(X_s) dx_s = \int_0^t \Delta_x^- f(X_s) dx_s \quad \mbox{ u.c.p.},$$
where u.c.p. means \textit{uniformly on compacts in probability}.

\end{enumerate}
\end{proposition}

\begin{proof}
\hspace{1cm}
\begin{enumerate}[leftmargin=0cm,itemindent=.5cm,labelwidth=\itemindent,labelsep=0cm,align=left]

\item It follows easily from standard results in mollification theory.\\

\item The first limit follows from Proposition \ref{prop:time_derivative_mollifier}. Moreover, one can easily notice
\begin{align*}
\int_0^t \Delta_t f_n(Y_s) ds &= \int_0^t \int_{\bR} \rho_n(-\xi) \Delta_t F(Y_s,\xi) d\xi ds \\
&= \int_{\bR}  \rho_n(-\xi) \left( \int_0^t  \Delta_t F(Y_s,\xi) ds \right) d\xi.
\end{align*}
Therefore,
$$\lim_{n \to +\infty} \int_0^t \Delta_t f_n(Y_s) ds = \int_0^t \Delta_t F(Y_s,0) ds = \int_0^t \Delta_t f(Y_s) ds.$$

\item Notice that
\begin{align}
\Delta_x f_n(Y_t) = \int_{\rho_{\min}}^{\rho_{\max}} \rho(\xi)\partial_h^- F\left(Y_t, - \frac{\xi}{n} \right) d\xi. \label{eq:deriv_Fn_int}
\end{align}
The boundedness assumptions means there exists an $x$-integrable process $(\psi_s)_{s \in [0,T]}$ such that
$$\max_{h \in [\rho_{\min}, \rho_{\max}]} \left|\partial_h^- F(X_s,-h) \right| \leq \psi_s.$$
Hence,
\begin{align*}
|\Delta_x f_n(Y_t)| &\leq \int_{\rho_{\min}}^{\rho_{\max}} \rho(\xi) \left|\partial_h^- F\left(Y_t, - \frac{\xi}{n} \right) \right| d\xi\leq  \int_{\rho_{\min}}^{\rho_{\max}} \rho(\xi) \psi_s d\xi = \psi_s.
\end{align*}
Therefore, by the Dominated Convergence Theorem for stochastic integrals, see \cite[Theorem 32, Chapter IV]{Protter2005}, we have the desired convergence.
\end{enumerate}
\end{proof}

\subsection{The Functional Meyer-Tanaka Formula}

We start this section by stating the assumptions on the functional $f$ such that the Meyer-Tanaka formula will hold.

\begin{hypotheses}\label{hypo:hypotheses}
\hspace{1cm}
\begin{enumerate}

\item $f$ $h$-time functional differentiable as in Definition \ref{def:time_derivative_mollifier};\\

\item $f$ and $\Delta_t F$ are $\Lambda$-$\phi$-equicontinuous as in Definition \ref{defi:assump_continuity};\\

\item $\partial_y^- \cF(Y_t,y)$ exists and is of bounded variation for $(t,y)$ jointly and for $y$ separately, for any $Y_t \in \Lambda$.\\

\item $(\partial_h^- F(X_t,-h))_{t \in [0,T]}$ is bounded in $h \in [\rho_{\min}, \rho_{\max}]$ by an $x$-integrable process.

\end{enumerate}

\end{hypotheses}

We are ready then to prove the main result of this paper.

\begin{theorem}[Functional Meyer-Tanaka Formula]\label{thm:meyer_tanaka_form}
Suppose $f$ satisfies Hypotheses \ref{hypo:hypotheses} and let $x$ be a continuous semimartingale. Then, the \textit{functional Meyer-Tanaka formula} holds
\begin{align}
f(X_t) &= f(X_0) + \int_0^t \Delta_t f(X_s) ds + \int_0^t \Delta_x^- f(X_{s}) dx_s \label{eq:meyer_tanaka_form} \\
&+ \int_{\bR}  L^x(t,y) d_y\partial_y^- \cF(X_t,y) - \int_0^t \int_{\bR} L^x(s,y) d_{s,y}\partial_y^- \cF(X_s,y) \quad \bP-\mbox{a.s.} \nonumber
\end{align}

\end{theorem}

\begin{proof}

As we studied in Section \ref{sec:mollification}, $f_n$ belongs to $\bC^{1,\infty}$ and by the functional It\^o formula, Theorem \ref{thm:fif}, we find
\begin{align*}
f_n(X_t) = f_n(X_0) +  \int_0^t &\Delta_t f_n(X_s) ds + \int_0^t \Delta_x f_n(X_s) dx_s + \frac{1}{2} \int_0^t \Delta_{xx} f_n(X_s) d\langle x \rangle_s.
\end{align*}
Moreover, by what was shown in Proposition \ref{prop:conv_properties}, the following convergences hold
\begin{align}
&\lim_{n \to +\infty} f_n(Y_t) = f(Y_t), \label{eq:converg_F} \\
&\lim_{n \to +\infty} \int_0^t \Delta_t f_n(Y_s) ds = \int_0^t \Delta_t f(Y_s) ds,\label{eq:converg_int_DeltatF}\\
&\lim_{n \to +\infty} \int_0^t \Delta_x f_n(X_s) dx_s = \int_0^t \Delta_x^- f(X_s) dx_s \quad \mbox{ u.c.p.} \label{eq:converg_int_DeltaxF}
\end{align}
for any $Y_t \in \Lambda$. Let us now analyse the It\^o term. Remember $\cF$ is defined by Equation (\ref{eq:cF}). If we denote the mollification of $\cF$ with respect to the $y$ variable by $\cF_n$, we can easily conclude, by Equation (\ref{eq:relation_f_F_cF}),
$$F_n(Y_{t-}^y,h) = \cF_n(Y_t,y+h)$$
and then
\begin{align*}
\partial_h^{(k)}  F_n(Y_{t-}^y,h) = \partial_y^{(k)}  \cF_n(Y_t,y+h).
\end{align*}
In particular, $\Delta_x^{(k)} f_n(Y_{t-}^y) = \partial_h^{(k)}  F_n(Y_{t-}^y,0) = \partial_y^{(k)}  \cF_n(Y_t,y)$. So, by Equation (\ref{eq:occup_time_functional}),
\begin{align*}
&\frac{1}{2} \int_0^t \Delta_{xx} f_n(X_s) d\langle x \rangle_s = \int_{\bR} \left( \int_0^t \partial_{yy} \cF_n(X_s,y) d_s L^x(s,y) \right) dy\\
&= \int_{\bR} \partial_{yy}\cF_n(X_t,y) L^x(t,y) dy -  \int_0^t \int_{\bR}  \Delta_t \partial_{yy}\cF_n(X_s,y) L^x(s,y) dy ds,
\end{align*}

Hence, for $g: \bR \longrightarrow \bR$ and $q: \bR^2 \longrightarrow \bR$ smooth and compactly supported, we have, by the computations performed in Section \ref{sec:int_by_parts},
\begin{align}
&\int_{\bR} \partial_{yy} \cF_n(Y_t,y) g(y) dy = -\int_{\bR}  g'(y)  \partial_y \cF_n(Y_t,y)dy \label{eq:local_time_int_1}\\
&\stackrel{n \to +\infty}{\longrightarrow} -\int_{\bR}  g'(y) \partial_y^- \cF(Y_t,y)dy =\int_{\bR}  g(y) d_y\partial_y^- \cF(Y_t,y), \nonumber
\end{align}
and
\begin{align}
\int_0^t \int_{\bR} \Delta_t \Delta_{xx} &f_n(Y_{s-}^y) q(s,y) dy ds  =  -  \int_{\bR} \partial_y \cF_n(Y_s,y) \partial_y q(s,y) \Big|_0^t dy  \label{eq:local_time_int_2}\\
&+ \int_0^t \int_{\bR} \partial_y \cF_n(Y_s,y)  \partial_{ty} q(s,y) dy ds \nonumber\\
&\stackrel{n \to +\infty}{\longrightarrow} -  \int_{\bR} \partial_y^- \cF(Y_s,y) \partial_y q(s,y) \Big|_0^t dy \nonumber\\
&+ \int_0^t \int_{\bR} \partial_y^-  \cF(Y_s,y)  \partial_{ty} q(s,y) dy ds \nonumber\\
&= \int_0^t \int_{\bR} q(s,y) d_{s,y} \partial_y^- \cF(Y_s,y), \nonumber
\end{align}
where the last equalities in (\ref{eq:local_time_int_1}) and (\ref{eq:local_time_int_2}) follow from item 3 of Hypotheses \ref{hypo:hypotheses}. Therefore, using well-known arguments along the lines of \cite[Proof of Theorem 2.1]{generalized_ito_elworthy}, we can extend the formulas above for $g(y) = L^x(t,y)$ and $q(s,y) = L^x(s,y)$, and finally conclude
\begin{align*}
\lim_{n \to +\infty} \frac{1}{2} \int_0^t \Delta_{xx} f_n(X_s) d\langle x \rangle_s &= \int_{\bR}  L^x(t,y) d_y\partial_y^- \cF(X_t,y) \\
&- \int_0^t \int_{\bR} L^x(s,y) d_{s,y}\partial_y^- \cF(X_s,y),
\end{align*}
as desired.
\end{proof}

\begin{remark}
By the same arguments presented in \cite{generalized_ito_elworthy_corrected}, we could show that $\int_{\bR}  L^x(t,y)$ $d_y\partial_y^- \cF(X_t,y)$ is of bounded variation in $t$ in $[0,T]$. Therefore, $(f(X_t))_{t \in [0,T]}$ is a semimartingale.
\end{remark}

\begin{remark}
Following the idea of \cite[Theorem 2.3]{generalized_ito_elworthy}, we could consider the process $x_t^{\star} = x_t - a_t$, where $(a_t)_{t \geq 0}$ is a continuous process of finite variation. It is obvious that $x^{\star}$ is also a semimartingale. Denote the local time of $x^{\star}$ by $L^{x-a}$. Therefore, the same argument of \cite[Theorem 2.3]{generalized_ito_elworthy} applied to the computation we have just performed in (\ref{eq:local_time_int_1}) and (\ref{eq:local_time_int_2}) gives us the following version of the functional Meyer-Tanaka formula
\begin{align}
f(X_t) &= f(X_0) + \int_0^t \Delta_t f(X_s) ds + \int_0^t \Delta_x^- f(X_{s}) dx_s \label{eq:meyer_tanaka_form2} \\
&+ \int_{\bR}  L^{x-a}(t,y) d_y\partial_y^- \cF(X_t,y+a_t) - \int_0^t \int_{\bR} L^{x-a}(s,y) d_{s,y}\partial_y^- \cF(X_s,y+a_s). \nonumber
\end{align}
This version of the formula will be used in the running maximum example in Section \ref{sec:running_max}.
\end{remark}


\section{Applications}\label{sec:applications}

\subsection{Convex Functionals}

In this section we define the notion of convexity for functionals and then discuss some of its basic properties. The main interesting consequence is that some of the technical assumptions in Hypotheses \ref{hypo:hypotheses} can be weakened.

\begin{definition}[Convex Functionals]\label{def:convex_funbctional}
We say $f$ is a \textit{convex functional} if $F(Y_t,\cdot)$ is a convex real function for any $Y_t \in \Lambda$.
\end{definition}

Notice that, for $f \in \bC^{1,2}$, convexity implies that $\Delta_{xx} f(Y_t) \geq 0$, for any $Y_t \in \Lambda$.

\begin{remark}
Another possible definition for convexity of a functional would be
\begin{align}\label{eq:strong_convex}
f(\lambda Y_t + (1 - \lambda)Z_t) \leq \lambda f(Y_t) + (1 - \lambda) f(Z_t),
\end{align}
for all $\lambda \in [0,1]$ and $Y_t,Z_t \in \Lambda$. Observe $Y_t$ and $Z_t$ must be in the same $\Lambda_t$ space. This clearly implies the previous definition of convexity because
$$F(Y_t, \lambda h_1 + (1-\lambda)h_2) = f(\lambda Y_t^{h_1} + (1 - \lambda) Y_t^{h_2}).$$
However, condition (\ref{eq:strong_convex}) is stronger than necessary for what follows.
\end{remark}

For a convex functional $f$, for any $Y_t \in  \Lambda$, $F(Y_t,\cdot)$ is continuous, $\partial_h^{\pm} F(Y_t,h)$ exist for any $h \in \bR$ and is non-decreasing in $h$. Moreover, $\partial_h^{\pm} F(Y_t,h) = \Delta_x^{\pm} F(Y_t,h)$, where these one-sided functional derivatives are obviously defined as
$$\Delta_x^{\pm} f(Y_t) = \lim_{h \to 0^{\pm}} \frac{f(Y_t^h) - f(Y_t)}{h}.$$

\begin{proposition}
\begin{enumerate}[leftmargin=0cm,itemindent=.5cm,labelwidth=\itemindent,labelsep=0cm,align=left]

Assume $f$ is convex. The following facts hold true:\\

\item $f_n$ is convex. Moreover, $\Delta_x f_n(Y_t)$ increasingly converges to $\Delta_x^- f(Y_t)$.\\

\item If $(\partial_h^- F(X_s,-h))_{s \in [0,T]}$ is $x$-integrable for $h=\rho_{\min}$ and $h=\rho_{\max}$, then
\begin{align*}
\lim_{n \to +\infty} \int_0^t \Delta_x f_n(X_s) dx_s = \int_0^t \Delta_x^- f(X_s) dx_s \ \mbox{ u.c.p.}
\end{align*}

\end{enumerate}
\end{proposition}

\begin{proof}
\hspace{1cm}
\begin{enumerate}[leftmargin=0cm,itemindent=.5cm,labelwidth=\itemindent,labelsep=0cm,align=left]

\item Indeed,
\begin{align*}
F_n(Y_t, \lambda h_1 + (1-\lambda)h_2) &= \int_{\bR} F(Y_t,(\lambda h_1 + (1-\lambda)h_2)-y) \rho_n(y) dy \\
&= \int_{\bR} F(Y_t,(\lambda (h_1 - y) + (1-\lambda)(h_2-y)) \rho_n(y) dy \\
&\leq \lambda F_n(Y_t, h_1) + (1-\lambda)F_n(Y_t, h_2).
\end{align*}
Hence, since $f_n$ is smooth, $\Delta_{xx} f_n(Y_t) \geq 0$. The second affirmation follows from:
\begin{align}
\Delta_x f_n(Y_t) = \int_{\rho_{\min}}^{\rho_{\max}} \rho(\xi)\partial_h^- F\left(Y_t, - \frac{\xi}{n} \right) d\xi, \label{eq:deriv_Fn_int_convex}
\end{align}
and it is easy to see the desired result using the fact that $\partial_h^- F(Y_t,h)$ is non-decreasing in $h$, because of the convexity of $f$.\\

\item Since $\partial_h^- F(Y_t,h)$ is non-decreasing in $h$, by Equation (\ref{eq:deriv_Fn_int_convex}),
\begin{align*}
&\partial_h^- F(Y_s,-\rho_{\max}) \leq \partial_h^- F\left(Y_s,- \frac{\rho_{\max}}{n} \right) \leq \partial_h^- F\left(Y_s,- \frac{\xi}{n} \right), \\
&\partial_h^- F\left(Y_s,- \frac{\xi}{n} \right) \leq \partial_h^- F\left(Y_s,- \frac{\rho_{\min}}{n} \right) \leq \partial_h^- F(Y_s, -\rho_{\min}),
\end{align*}
where we are using the fact that $\rho_{\min} < 0$ and $\rho_{\max} > 0$; see Remark \ref{rmk:bound_mollifier}. Hence
$$|\Delta_x f_n(Y_s)| \leq |\partial_h^- F(Y_s, -\rho_{\min})| + |\partial_h^- F(Y_s,-\rho_{\max})|,$$
and the convergence follows as in Proposition \ref{prop:conv_properties}

\end{enumerate}
\end{proof}

Therefore, we might then consider the following class of convex functionals, where we have weakened conditions 3 and 4 of Hypotheses \ref{hypo:hypotheses}:

\begin{hypotheses}\label{hypo:hypotheses_convex}
\hspace{1cm}
\begin{enumerate}

\item $f$ $h$-time functional differentiable as in Definition \ref{def:time_derivative_mollifier};\\

\item $f$ and $\Delta_t F$ are $\Lambda$-$\phi$-equicontinuous as in Definition \ref{defi:assump_continuity};\\

\item $\partial_y^- \cF(Y_s,y)$ is of bounded variation for $(s,y)$ jointly, for any $Y \in \Lambda$;\\

\item $(\partial_h^- F(X_s,-h))_{s \in [0,T]}$ is $x$-integrable for $h=\rho_{\min}$ and $h=\rho_{\max}$;\\

\item $f$ is convex;

\end{enumerate}

\end{hypotheses}

\begin{remark}
It is straightforward to notice that if $f$ satisfies Hypotheses \ref{hypo:hypotheses_convex}, then $f$ also satisfies Hypotheses \ref{hypo:hypotheses}. So, the functional Meyer-Tanaka formula, Theorem \ref{thm:meyer_tanaka_form}, holds for $f$.
\end{remark}

Similarly as in \cite{Protter2005}, we may analyze the limit of $f_n(Y_t)$ without identifying the limit of the It\^o term.

\begin{theorem}
Let $f$ be a functional satisfying Hypotheses \ref{hypo:hypotheses_convex} and $x$ a continuous semimartingale. Then
\begin{align} \label{eq:ito_formula_A}
f(X_t) = f(X_0) + \int_0^t \Delta_t f(X_s) ds + \int_0^t \Delta_x^- f(X_s) dx_s + \frac{1}{2} A_t^f \quad \bP-\mbox{a.s.},
\end{align}
where $A_t^f$ is a continuous and increasing process.
\end{theorem}

\begin{proof}
As we have seen in the proof of Theorem \ref{thm:meyer_tanaka_form},
\begin{align*}
f_n(X_t) = f_n(X_0) +  \int_0^t &\Delta_t f_n(X_s) ds + \int_0^t \Delta_x f_n(X_s) dx_s + \frac{1}{2} \int_0^t \Delta_{xx} f_n(X_s) d\langle x \rangle_s.
\end{align*}
Consider now the continuous process
$$A_t^n = \int_0^t \Delta_{xx} f_n(X_s) d\langle x \rangle_s.$$
This process is increasing because $f_n$ is convex, which means $\Delta_{xx} f_n \geq 0$. Hence, by Equations (\ref{eq:converg_F})--(\ref{eq:converg_int_DeltatF}), $A_t^n$ converges u.c.p. to a continuous increasing process $A_t^f$ that satisfies Equation (\ref{eq:ito_formula_A}).
\end{proof}

\begin{remark}
As in the classical case, Equation (\ref{eq:ito_formula_A}) shows that the convex functional of a continuous semimartingale is also a continuous semimartingale.
\end{remark}

\subsection{The Running Maximum}\label{sec:running_max}

The running maximum (or more precisely, supremum) is defined as
\begin{align}\label{eq:running_max_application}
\barm(Y_t) = \sup_{0 \leq s \leq t} y_s,
\end{align}
for any $Y_t \in \Lambda$.\\

Let us first verify that $\barm$ is $\Lambda$-continuous. Notice $\barm(Y_t) = \barm(Y_{t,r})$, for any $Y_t \in \Lambda$ and $r \geq 0$. Hence, if we fix $Y_t, Z_s \in \Lambda$ with $s \geq t$, we find
\begin{align*}
|\barm(Y_t) - \barm(Z_s)| &= |\barm(Y_{t,s-t}) - \barm(Z_s)| \\
&= \left| \sup_{0 \leq u \leq s} Y_{t,s-t}(u) - \sup_{0 \leq u \leq s} Z_s(u)\right| \\
&\leq \sup_{0 \leq u \leq s} | Y_{t,s-t}(u)  - Z_s(u)| \leq d_{\Lambda}(Y_t,Z_s).
\end{align*}
Therefore, the running maximum is (Lipschitz) $\Lambda$-continuous. Moreover, one could also verify that $\Delta_t \barm(Y_t) = 0$. Define now the subset of $\Lambda$ where the supremum is attained at the last value:
$$\cS = \left\{Y_t \in \Lambda \ ; \ \barm(Y_t) = y_t \right\}.$$
For paths in $\cS$, the space functional derivative is not defined: the right derivative is 1 and the left derivative is 0. For paths outside $\cS$, the space functional derivative is well-defined and it is 0: $\Delta_x \barm(Y_t) = 0$, for $Y_t \notin \cS$.\\

We show below that the running maximum is $\Lambda$-equicontinuous according to Definition \ref{defi:assump_continuity}. One can easily see that
$$\barM(Y_t,\xi) = \barm(Y_t^\xi) = \max\left\{ \barm(Y_t), y_t + \xi^+\right\}.$$
Moreover, for any $Y_t, Z_s \in \Lambda$,
\begin{align*}
|\barM(Y_t,\xi) - \barM(Z_s,\xi)| &\leq d_{\Lambda}(Y_t^{\xi},Z_s^{\xi}) = d_{\Lambda}(Y_t, Z_s).
\end{align*}
Since the bound above is independent of $\xi$, the running maximum is $\Lambda$-equicontinuous. Besides, we notice that
$$\barm((Y_{t,\delta t})^\xi) = \max\left\{ \barm(Y_t), y_t + \xi^+\right\} = \barm((Y_t^\xi)_{t,\delta t}),$$
and therefore, $\barm$ satisfies Definition \ref{def:time_derivative_mollifier}. Furthermore, this shows that the running maximum is locally weakly path-dependent, i.e. the Lie bracket is zero (in the limit characterization), see Remark \ref{rmk:lie_bracket}.\\

Additionally, one can easily prove that the running maximum $\barm(Y_t)$ is a (non-smooth) convex functional. It is actually convex in the stronger sense of \eqref{eq:strong_convex}. Additionally, $\partial_h^- \barM(Y_t,h) = \Delta_x^- \barm(Y_t^h) = 0$.\\

Finally, we are ready to employ the functional Meyer-Tanaka formula, Theorem \eqref{thm:meyer_tanaka_form}, to the running maximum. Firstly, we have already shown that
\begin{align*}
&\Delta_t \barm(Y_t) = 0 \mbox{ and }\Delta_x^- \barm(Y_t) = 0, \ \forall \ Y_t \in \Lambda.
\end{align*}
Notice now
$$\barm(Y_{t-}) = \sup_{0 \leq s < t} y_u$$
(time $t$ is not allowed in the supremum) and notice that
$$\overline{\cM}(Y_t,y+h) = \barm(Y_{t-}^{y+h})= \max\{y+h, \barm(Y_{t-})\}.$$
Hence, we can compute
\begin{align*}
\partial_y^- \overline{\cM}(Y_t,y) = 1_{\{y > \barm(Y_{t-})\}} \Rightarrow d_y \partial_y^- \overline{\cM}(Y_t,y) = \delta_{\barm(Y_{t-})}(dy),
\end{align*}
where $\delta_c$ is the Dirac mass concentrated at $c \in \bR$. We then face a problem, because $d_{t,y}\partial_y^- \overline{\cM}(Y_t,y)$ is not easily computed. However, we notice that
\begin{align*}
\partial_y^- \overline{\cM}(Y_t,y+\barm(Y_{t-})) = 1_{\{y > 0\}} \Rightarrow & \ d_y \partial_y^- \overline{\cM}(Y_t,y+\barm(Y_{t-})) = \delta_0(dy) \\
&\mbox{ and } d_{t,y} \partial_y^- \overline{\cM}(Y_t,y+\barm(Y_{t-})) = 0.
\end{align*}
Hence, we have seen that $\barm$ satisfies Hypotheses \ref{hypo:hypotheses_convex}. We will then apply formula (\ref{eq:meyer_tanaka_form2}) with $a_t = \barm(X_t) = \barm(X_{t-})$, which is clearly a continuous process of finite variation, since $(x_t)_{t \geq 0}$ is a continuous semimartingale. These equalities hold because the process $x$ is continuous. Therefore, by Equation (\ref{eq:meyer_tanaka_form2}), we finally find the pathwise version of the important formula of L\'evy:
\begin{align*}
\barm_t = \max_{0 \leq s \leq t} x_s = x_0 + L^{x-\barm}(t,0),
\end{align*}
where $L^{x-\barm}$ is the local time of the process $(x_t - \barm_t)_{t \in [0,T]}$.\\

Furthermore, the same analysis could be performed for the running minimum. Indeed,
\begin{align}\label{eq:running_min_application}
\underline{m}(Y_t) = \inf_{0 \leq s \leq t} y_s = -\barm(-Y_t),
\end{align}
where $-Y_t(u) = -y_u$, for all $u \leq t$. Therefore, $\underline{m}$ satisfies Hypotheses \ref{hypo:hypotheses_convex} as well and
$$\underline{\cM}(Y_t, y) = - \overline{\cM}(-Y_t, -y) \Rightarrow \partial_y^- \underline{\cM}(Y_t, y) = 1_{\{y < \underline{m}(Y_{t-})\}},$$
where
$$\underline{m}(Y_{t-}) = \inf_{0 \leq s < t} y_u.$$
Then,
\begin{align*}
\underline{m}_t = \min_{0 \leq s \leq t} x_s = x_0 -  L^{x - \underline{m}}(t,0),
\end{align*}

\subsection{Characterization of Local Martingales Functions of the Running Maximum}

In the articles \cite{Obloj2006, Obloj2006a}, the authors studied the problem of complete characterization of local martingales that are functions of the current state of a continuous local martingale and its running maximum. In this section, we will show how the functional It\^o calculus framework can be used to study this problem.\\

\begin{theorem}\label{thm:local_time_zero_delta_t}
Let $(x_t)_{t \geq 0}$ be a continuous local martingale and consider a functional $f$ satisfying Hypotheses \ref{hypo:hypotheses}. Then $(f(X_t))_{t \geq 0}$ is a local martingale if and only if
\begin{align}\label{eq:local_time_zero_delta_t}
\int_0^t \Delta_t f(X_s) ds &+ \int_{\bR} L^x(t,y) d_y \partial_y^- \cF(X_t,y) \\
&= \int_0^t \int_{\bR} L^x(s,y) d_{s,y} \partial_y^- \cF(X_s,y),  \ \forall \ t \geq 0, \quad \bP-\mbox{a.s. } \nonumber
\end{align}
hen, if $\Delta_t f(X_s) = 0, \ \forall \ s \geq 0,$ $\bP$-a.s. and if $\partial_y  \partial_y^- \cF(X_s,y)$ exists, Equation (\ref{eq:local_time_zero_delta_t}) is equivalent to
\begin{align}\label{eq:local_time_zero_delta_t_simple}
\int_{\bR} \int_0^t \partial_y \partial_y^- \cF(X_s,y) d_s L^x(s,y) dy = 0.
\end{align}
\end{theorem}

\begin{proof}
By the functional Meyer-Tanaka formula, Equation (\ref{eq:meyer_tanaka_form}), we find
\begin{align*}
f(X_t) &= f(X_0) + \int_0^t \Delta_t f(X_s) ds + \int_0^t \Delta_x^- f(X_{s}) dx_s\\
&+ \int_{\bR}  L^x(t,y) d_y\partial_y^- \cF(X_t,y) - \int_0^t \int_{\bR} L^x(s,y) d_{s,y}\partial_y^- \cF(X_s,y). \nonumber
\end{align*}
Hence, $(f(X_t))_{t \geq 0}$ is a local martingale if and only if
\begin{align*}
\int_0^t \Delta_t f(X_s) ds &+ \int_{\bR} L^x(t,y) d_y \partial_y^- \cF(X_t,y) = \int_0^t \int_{\bR} L^x(s,y) d_{s,y} \partial_y^- \cF(X_s,y).
\end{align*}
Furthermore, Equation (\ref{eq:local_time_zero_delta_t_simple}) follows from
\begin{align*}
&\int_0^t \int_{\bR} L^x(s,y) d_{s,y} \partial_y^- \cF(X_s,y) = \int_{\bR} \int_0^t L^x(s,y) d_s \partial_y \partial_y^- \cF(X_s,y) dy\\
&= \int_{\bR} \left(L^x(t,y) \partial_y \partial_y^- \cF(X_t,y) - \int_0^t \partial_y \partial_y^- \cF(X_s,y) d_s L^x(s,y) \right) dy.
\end{align*}
\end{proof}

\begin{remark}\label{rmk:support}
Let $(x_t)_{t \geq 0}$ be a continuous local martingale. Denote
$$C = \bigcup_{t > 0} supp(\barm_t) \subset \bR,$$
where $supp(z)$ is the support of the random variable $z$. By the Dambis-Dubins-Schwarz Theorem \cite[Theorem 4.6, Chapter 3]{Karatzas1988}, $x_t = b_{\langle x \rangle_t}$, where $b$ is a Brownian motion, which implies that $C = \bR_+$ for any continuous local martingale. This will be useful in the proof of the next theorem. 
\end{remark}

\begin{theorem}\label{thm:obloj}
Let $(x_t)_{t \geq 0}$ be a continuous local martingale with $x_0 = 0$ and consider $H: \bR^2 \longrightarrow \bR$ in $C^1(\bR^2)$. Then $(H(x_t, \barm_t))_{t \geq 0}$ is a right-continuous local martingale in the natural filtration of $x$ if and only if there exists $\psi: \bR \longrightarrow \bR$ in $C^1(\bR)$ such that
\begin{align}
H(x_1, x_2) = \int_0^{x_2} \psi(s) ds - \psi(x_2)(x_2 - x_1) + H(0,0), \ \forall \ (x_1, x_2) \in \bR^2. \label{eq:obloj}
\end{align}
\end{theorem}

\begin{proof}
We start by defining the functional $f(Y_t) = H(y_t, \barm(Y_t))$. \\

Since $\barm(Y_{t,\delta t}) = \barm(Y_t)$ and $\Delta_x^- \barm(Y_t) = 0$, we easily conclude that $\Delta_t f(Y_t) = 0$ and $\Delta_x^- f(Y_t) = \partial_1 H(y_t, \barm(Y_t))$, where $\partial_i$ denotes the derivative with respect to $i$th variable of $H$, $i=1,2$. Smoothness of $H$ implies that $f$ satisfies Hypotheses \ref{hypo:hypotheses}.

To ease the burden of notation, notice that $\barm(X_{t-}) = \barm(X_t) = \barm_t$, since $x$ is continuous almost surely. By Theorem \ref{thm:local_time_zero_delta_t}, $(f(X_t))_{t \geq 0}$ is a local martingale if and only if
\begin{align}\label{eq:local_time_zero}
\int_{\bR} \int_0^t \partial_y \partial_y^- \cF(X_t,y+\barm_t) d_s L^x(s,y) dy = 0, \ \forall \ t \geq 0, \quad \bP-\mbox{a.s. }
\end{align}
By a mollification argument, we may assume for the moment that $H \in C^2(\bR^2)$ and then we are able to directly compute $\partial_y \partial_y^- \cF(X_t,y+\barm_t)$.

Note that $\cF(X_t,y+\barm_t) = H(y+\barm_t, \max\{y+\barm_t, \barm_t\})$. Hence, since $\max\{y+\barm_t, \barm_t\} = \barm_t + y^+$, we find
\begin{align*}
\partial_y^- \cF(X_t,y+\barm_t) &= \partial_1 H(y+\barm_t, \barm_t + y^+) \\
&+ \partial_2 H(y+\barm_t, \barm_t + y^+) 1_{\{y > 0\}},
\end{align*}
which implies that
\begin{align*}
\partial_y \partial_y^- \cF(X_t,y+\barm_t) &= \partial_{11} H(y+\barm_t, \barm_t + y^+) \\
&+ 1_{\{y > 0\}} (2\partial_{12} + \partial_{22})  H(y+\barm_t, \barm_t + y^+) \\
&+ \partial_2 H(y+\barm_t, \barm_t + y^+) \delta_0.
\end{align*}
Therefore,
\begin{align*}
&\int_{\bR} \int_0^t \partial_y \partial_y^- \cF(X_t,y+\barm_t) d_s L^{x-\barm}(s,y) dy \\
&= \int_{\bR} \int_0^t \partial_{11} H(y+\barm_t, \barm_t + y^+) d_s L^{x-\barm}(s,y) dy \\
&+ \int_0^{+\infty} \int_0^t (2\partial_{12} + \partial_{22})  H(y+\barm_t, \barm_t + y) d_s L^{x-\barm}(s,y) dy \\
&+ \int_0^t  \partial_2 H(\barm_t, \barm_t) d_s L^{x-\barm}(s,0) \\
&= \int_{-\infty}^0 \int_0^t \partial_{11} H(y+\barm_t, \barm_t + y^+) d_s L^{x-\barm}(s,y) dy \\
&+ \int_0^t  \partial_2 H(\barm_t, \barm_t) d_s L^{x-\barm}(s,0)
\end{align*}
since $L^{x-\barm}(t,y) = 0$, for $y > 0$. Then, by Equation (\ref{eq:local_time_zero}) and Remark \ref{rmk:support} and since the measures $d_s L^{x-\barm}(s,y)$ have disjoint supports for different $y < 0$, we must have, for all $(x_1, x_2) \in \bR \times \bR_+$,
\begin{align}\label{eq:condition_H}
\partial_{11} H(x_1, x_2) = 0 \mbox{ and } \partial_2 H(x_2,x_2) = 0 .
\end{align}
These equations can be solved analytically. The first equation above implies there exists $\psi, \varphi \in C^1(\bR)$ such that
$$H(x_1, x_2) = \psi(x_2) x_1 + \varphi(x_2).$$
Then, by the first equation in (\ref{eq:condition_H}) we find that
$$\psi'(x_2) x_2 + \varphi'(x_2) = 0,$$
which means
$$\varphi(x_2) = \varphi(0) - \int_0^{x_2} \psi'(s)sds = \varphi(0) - \psi(x_2)x_2 + \int_0^{x_2} \psi(s)ds.$$
Moreover, notice that $\varphi(0) = H(0,0)$. Therefore, a function $H \in C^1(\bR^2)$ is such that $(H(x_t, \barm_t))_{t \geq 0}$ is a local martingale if and only if there exists $\psi  \in C^1(\bR)$ such that
\begin{align*}
H(x_1, x_2) = \int_0^{x_2} \psi(s) ds - \psi(x_2)(x_2 - x_1) + H(0,0), \ \forall \ (x_1, x_2) \in \bR^2.
\end{align*}
\end{proof}

\begin{remark}
It is proved in \cite{Obloj2006} that
$$\psi(\xi) = \left.\frac{d\langle x, H(x, \barm) \rangle_t}{d\langle x \rangle_t}\right|_{t = T_\xi},$$
where $T_\xi = \inf\{t; x_t = \xi\}$. Within the functional framework, it easy to see that
\begin{align}
\psi(\barm(Y_t)) = \Delta_x^- H(y_t, \barm(Y_t)).
\end{align}
This formula could be evaluated pathwise to find $\psi(\xi)$.
\end{remark}

\subsection{Quadratic Variation}

The functional Meyer-Tanaka formula, Theorem \ref{thm:meyer_tanaka_form}, could provide interesting results even when applied to smooth functionals. As an illustrative example, let us consider the quadratic variation functional $\QV$, see \cite{fito_generalization} for the proper pathwise definition and discussion on its smoothness. It is straightforward and intuitive that $\Delta_t \QV = 0$ and that $\QV(Y_{t-}^y) = \QV(Y_{t-}) + (y-y_{t-})^2$. Therefore, the functional Meyer-Tanaka formula gives us the well-known formula
$$\langle x \rangle_t = 2\int_\bR L^x(t,y) dy,$$
for any continuous semimartingale $x$.

\subsection{Increasing Functionals}

\begin{definition}
A functional $f: \Lambda \longrightarrow \bR$ is called \textit{increasing} if $f(Y_t) \geq f(Y_s)$, for all $Y_t \in \Lambda$ and $s \leq t$, where $Y_s$ is the restriction of $Y_t$ to $[0,s]$.
\end{definition}

Consider now an increasing functional $f$ in $\bC^{1,2}$ with $\Delta_x f \in \bC^{1,1}$. Then, we find that $\Delta_t f \geq 0$ and that the path $(f(Y_t))_{t \in [0,T]}$ is of finite variation for any $Y_T \Lambda_T$. Hence, if $(w_t)_{t \in [0,T]}$ is a Brownian motion in $(\Omega, \fF, \bP)$, by the Functional It\^o Formula,
$$f(W_t) = f(W_0) + \int_0^t \Delta_t f(W_u) du + \int_0^t \Delta_x f(W_u) dw_u + \frac{1}{2} \int_0^t \Delta_{xx} f(W_u) du.$$
Now, since the increasing process $(f(W_t))_{t \geq 0}$ is of finite variation, by the uniqueness of the semimartingale decomposition, we conclude that $\Delta_x f(W_u) = 0$, for $u \in [0,T]$. Since $\Delta_x f \in \bC^{1,1}$ and the support of Brownian paths is the set of continuous paths, we have $\Delta_{xx} f(Y_t) = 0$ for any continuous path $Y_t$, see \cite[Theorem 2.2]{fito_fournie_thesis}. Therefore,
$$f(W_t) = f(W_0) + \int_0^t \Delta_t f(W_u) du.$$
Furthermore, by the $\Lambda$-continuity of the functionals involved in the equality above, we conclude that
$$f(Y_t) = f(Y_0) + \int_0^t \Delta_t f(Y_u) du,$$
for any continuous path $Y_t$. What happens if the functional $f$ is not smooth, but satisfies Hypotheses \ref{hypo:hypotheses}? In this case, for any local martingale $x$,
\begin{align*}
f(X_t) &= f(X_0) + \int_0^t \Delta_t f(X_s) ds + \int_0^t \Delta_x^- f(X_s) dx_s \\
&+ \int_{\bR}  L^x(t,y) d_y\partial_y^- \cF(X_t,y) - \int_0^t \int_{\bR} L^x(s,y) d_{s,y}\partial_y^- \cF(X_s,y),
\end{align*}
and, for the same reason, the stochastic integral term vanishes and we conclude that
\begin{align*}
f(X_t) &= f(X_0) + \int_0^t \Delta_t f(X_s) ds + \int_{\bR}  L^x(t,y) d_y\partial_y^- \cF(X_t,y) \\
&- \int_0^t \int_{\bR} L^x(s,y) d_{s,y}\partial_y^- \cF(X_s,y).
\end{align*}

\section*{Acknowledgements}
Firstly, I express my gratitude to B. Dupire for proposing such interesting problem and for the helpful discussions. I am thankful to J.-P. Fouque and T. Ichiba for all the insightful comments. Part of the research was carried out in part during the summer internship of 2013 supervised by B. Dupire at Bloomberg LP.

\bibliographystyle{plain}

\end{document}